\documentclass{article}

\usepackage[english]{babel}

\usepackage[a4paper]{geometry}

\usepackage{amsfonts, amsmath, amssymb, amstext, amsthm}

\usepackage[shortlabels]{enumitem}
\usepackage{float}
\usepackage{graphicx}
\usepackage[colorlinks=true, allcolors=blue]{hyperref}
\usepackage{cleveref}
\usepackage{listings}
\usepackage{mathrsfs}
\usepackage{mathtools}
\usepackage{multicol}
\usepackage{wasysym}
\usepackage{subcaption}

\usepackage{tikz}
\usetikzlibrary{arrows, backgrounds, calc, chains, decorations, patterns, positioning, shapes}

\theoremstyle{plain}

\theoremstyle{definition}
\newtheorem{theorem}{Theorem}[section]
\newtheorem{conjecture}[theorem]{Conjecture}

\newtheorem{definition}[theorem]{Definition}
\newtheorem{lemma}[theorem]{Lemma}
\newtheorem{proposition}[theorem]{Proposition}

\theoremstyle{remark}
\newtheorem{remark}[theorem]{Remark}

\makeatletter%
\renewenvironment{proof}[1][\proofname]{%
	\par\pushQED{\qed}\normalfont%
	\topsep6\p@\@plus6\p@\relax
	\trivlist\item[\hskip\labelsep\bfseries#1\@addpunct{.}]%
	\ignorespaces
}{%
	\qedhere 
}
\makeatother

\makeatletter%
\DeclareRobustCommand*{\bfseries}{%
	\not@math@alphabet\bfseries\mathbf
	\fontseries\bfdefault\selectfont
	\boldmath
}
\makeatother

\renewcommand{\th}{^\mathsf{th}}

\newcommand{\area}{\mathsf{area}}
\newcommand{\inv}{\mathsf{inv}}
\newcommand{\height}{\mathsf{ht}}

\newcommand{\lv}{\mathsf{lv}}

\newcommand{\rev}{\mathsf{rev}}

\newcommand{\LD}{\mathsf{LD}} 
\newcommand{\zLD}{\mathsf{zLD}} 
\newcommand{\RTT}{\mathsf{RTT}} 
\newcommand{\zRTT}{\mathsf{zRTT}} 

\newcommand{\Ht}{\widetilde{H}}

\newcommand{\N}{\mathbb{N}}
\newcommand{\Q}{\mathbb{Q}}

\let\oldXi\Xi
\let\Xi\undefined
\DeclareMathOperator{\Xi}{\oldXi}
\DeclareMathOperator{\PPi}{\Pi}

\newcommand{\zeropaths}{\zLD^2(n)}

\newcommand{\zerotrees}{\zRTT_0(n)}

\newcommand{\comp}{\bumpeq}

\newcommand{\red}{\color{red}}
\newcommand{\blue}{\color{blue}}
\newcommand{\green}{\color{green!50!black}}

\title{A Proof of the Symmetric Theta Conjecture when $q=0$}

\author{Alessandra Caraceni, Alessandro Iraci}

\begin{document}

\maketitle

\begin{abstract}
    In \cite{DAdderioIraciLeBorgneRomeroVandenWyngaerd2022TieredTrees}, the authors conjecture a combinatorial formula for the expressions $\Xi e_\alpha \rvert_{t=1}$, known as \emph{Symmetric Theta Trees Conjecture}, in terms of tiered trees with an inversion statistic. In \cite{IraciRomero2022DeltaTheta}, the authors prove a combinatorial formula for the same symmetric function, in terms of doubly labelled Dyck paths with the area statistic. In this paper, we give an explicit bijection between the subsets of the two families of objects when the relevant statistic is equal to $0$, thus proving the Symmetric Theta Tree Conjecture when $q=0$.
\end{abstract}

\section{Introduction}

The \emph{Theta operators} $\Theta_f$, for any symmetric function $f$, are a family of operators on symmetric functions that show remarkable combinatorial properties. They were first defined in \cite{DAdderioIraciVandenWyngaerd2021ThetaOperators}, and they were instrumental in the formulation and proof of the \emph{compositional Delta conjecture} \cite{DAdderioMellit2022CompositionalDelta}, a refinement of the famous \emph{(rise) Delta conjecture} \cite{HaglundRemmelWilson2018DeltaConjecture}. The Delta conjecture is one of many combinatorial expressions tied to the combinatorial theory of Macdonald polynomials, which started with the shuffle theorem \cite{HaglundHaimanLoehrRemmelUlyanov2005ShuffleConjecture, CarlssonMellit2018ShuffleConjecture} and then inspired a wide variety of results.


\emph{Tiered trees} were first defined in \cite{DuganGlennonGunnellsSteingrimsson2019TieredTreesqEulerian} as trees on vertices labelled $1,\dots, n$ with an integer-valued level function $\lv$ on the vertices such that vertices labelled $i$ and $j$ with $i<j$ may be adjacent only if $\lv(i) < \lv(j)$.  They are a generalisation of \emph{intransitive trees} \cite{Postnikov1997IntransitiveTrees}, which are tiered trees with only two levels. The notion is related to spanning trees of inversion graphs and has connections to the \emph{abelian sandpile model} on such graphs \cite{DukesSeligSmithSteingrimsson2018TieredTrees}; in \cite{DAdderioIraciLeBorgneRomeroVandenWyngaerd2022TieredTrees}, the authors slightly extend the definition of tiered trees to allow for non-distinct labels. Tiered trees naturally arise as counting both absolutely irreducible representations of certain supernova quivers and certain torus orbits on partial flag varieties of type $A$ \cite{DuganGlennonGunnellsSteingrimsson2019TieredTreesqEulerian, GunnellsLetellierRodriguezVillegas2018Quivers}.

In \cite{DAdderioIraciLeBorgneRomeroVandenWyngaerd2022TieredTrees}, the authors conjecture a combinatorial formula for the expressions $\Theta_{e_\alpha} e_1 \rvert_{t=1}$, for any composition $\alpha \vDash n$, in terms of tiered trees with an inversion statistic; they also prove the special case $\alpha = 1^n$. We are interested in a similar conjectural formula appearing in the same paper, giving a combinatorial interpretation for $\Xi e_\alpha \coloneqq M \Delta_{e_1} \Pi e_\alpha^\ast$ when $t=1$ (the notation $\Xi$ is posterior), of which no special cases were known prior to our result. This conjecture, known as \emph{symmetric Theta conjecture}, turned out to have a significant tie to the combinatorial theory of Macdonald polynomials.

In \cite{IraciRomero2022DeltaTheta}, the authors formally introduce the $\Xi$ operator, and prove a combinatorial formula for $\Xi e_\alpha \rvert_{t=1}$ in terms of certain families of Dyck paths, with the classical area statistic. They also show an explicit $e$-expansion in terms of labelled Dyck paths, thus proving $e$-positivity of $\Xi e_\alpha \rvert_{t=1}$, and getting as a corlloary the univariate versions of the shuffle theorem, the Delta conjecture, and more. Later on, in \cite{BergeronHaglundIraciRomero2023SuperNablaOperator}, the authors refine the result giving a monomial expansion for the expression $\nabla_\ast e_n = \sum_{\lambda \vdash n} m_\lambda \otimes \Xi e_\lambda$ when $t=1$, which is independently symmetric in two sets of variables, in terms of $2$-labelled Dyck paths with the area statistic.

Indeed, combining these results, we get that the symmetric Theta conjecture also gives a (conjectural) monomial expansion for $\nabla_\ast e_n \rvert_{t=1}$, which can, in principle, be proved by finding an explicit bijection between tiered trees and $2$-labelled Dyck paths, as long as it maps the number of inversions of the tree to the area of the Dyck path, and preserves the associated monomial. In turns out that, even for $n=2$, preserving the individual pairs of labels is impossible.

Due to this difficulty, tackling this problem turned out to be harder than expected; however, the bijective approach, if appropriately modified, can still work. In this paper we exhibit an explicit bijection between tiered trees with no inversions and $2$-labelled Dyck paths with area equal to $0$, preserving the $X$-monomial and sending the $Y$-monomial to one with the same multiset of exponents; since $\nabla_\ast e_n \rvert_{t=1}$ is known to be symmetric, this proves the symmetric Theta conjecture when $q=0$.

This paper is structured as follows: in \Cref{sec:combinatorics}, we give the combinatorial definitions that are necessary to state the combinatorial expansions we are interested in; in \Cref{sec:symmetricfunctions}, we define the symmetric function operators we need, and then state the relevant combinatorial interpretations of the symmetric function expressions we want to study; in \Cref{sec:bijection}, we state the bijection and give a proof of the symmetric Theta conjecture when $q=0$.

\section{Combinatorial definitions}
\label{sec:combinatorics}

\subsection{Tiered trees}\label{sec:trees}

In this work, a \emph{graph} $G$ will be a pair $(V,E)$, with $V$ a finite set of \emph{vertices} and $E \subseteq \binom{V}{2}$ a set of \emph{edges} (hence no loops nor multiple edges are allowed). We say that $i,j \in V$ are \emph{neighbours} in $G$ if $\{i,j\} \in E$. We use the habitual notions of \emph{paths}, \emph{closed paths}, \emph{circuits}, \emph{connected components}, \emph{distance} between two vertices, and so on.
	
A \emph{forest} is a graph with no circuits; a \emph{tree} is a connected forest. Notice that a forest is a union of trees. A \emph{rooted tree} is a tree $(V,E)$ with a distinguished vertex $r \in V$ which we call its \emph{root}; we call a \emph{rooted forest} a disjoint union of rooted trees.

Rooted forests on $n$ vertices are naturally in bijection with rooted trees on $n+1$ vertices. Indeed, given a rooted forest $F$ with vertex set $V$ and edge set $E$, one can obtain a rooted tree $T$ with vertex set $V \cup \{r\}$, where $r$ is a new vertex which is going to be the root of the tree, and edge set $E \cup \{ \{r, r_1\}, \dots, \{r, r_k\} \}$, where $r_1, \dots, r_k$ are the roots of the connected components of $F$. Vice versa, given a rooted tree $T$, we can obtain a rooted forest by removing the root and rooting each connected component of the remaining forest at the vertex that used to be connected to the root of $T$.

Let $T$ be a rooted tree $(V,E,r)$ with root $r \in V$. Given a vertex $v \in V$, we define the \emph{height} of $v$ as the graph distance $\height(v)$ between $v$ and $r$. We define the \emph{parent} of $v \neq r$ as the unique neighbour $p(v)$ of $v$ such that $\height(p(v)) < \height(v)$, and we say that $v$ is a \emph{child} of $p(v)$. We say that $u$ is a \emph{descendant} of $v$ (and $v$ is an \emph{ancestor} of $u$) if there exists $k > 0$ such that $v = p^k(u)$.

\begin{definition}
	\label{def:tiered-rooted-tree}
	A \emph{tiered rooted forest} is a rooted forest $F = (V,E,r)$ equipped with two functions $w, \lv \colon V \setminus \{r\} \rightarrow \N_+$, called \emph{label} and \emph{level}, such that:	
	\begin{enumerate}
		\item if $\{ u, v \} \in E$, then $w(u) \neq w(v)$ and $\lv(u) \neq \lv(v)$;
		\item if $\{ u, v \} \in E$, then $w(u) < w(v) \iff \lv(u) < \lv(v)$;
		\item if $u \neq v$ and $p(u) = p(v)$, then $(w(u), \lv(u)) \neq (w(v), \lv(v))$.
	\end{enumerate}

    A \emph{tiered rooted tree} is a tree $T = (V,E,r)$ rooted at $r \in V$, with two functions $w, \lv \colon V \setminus \{r\} \rightarrow \N$ such that its corresponding forest is a tiered rooted forest, and such that $w(r) = \lv(r) = 0$.
\end{definition}

We define $\RTT_0(n)$ to be the set of tiered rooted trees on $n+1$ vertices (including the root).

It will be convenient later to have a shorthand for the \emph{reverse level}. Let $F$ be a tiered rooted tree, and let \[ L = \max_{v \in V \setminus \{r\}} \lv(v), \quad l = \min_{v \in V \setminus \{r\}} \lv(v). \] For $v \in V$, we set $\lv'(v) = L + l - \lv(v)$.

\begin{definition}
    We define the \emph{label composition} (resp.~\emph{level composition}) of a tiered rooted forest $F$ on $n$ vertices to be the weak composition $\alpha(F) \vDash_0 n$ (resp.~$\beta(F) \vDash_0 n$) defined by $\alpha_i \coloneqq w^{-1}(i)$ (resp.~$\beta_i \coloneqq \lv^{-1}(i)$); that is, $\alpha_i$ (resp.~$\beta_i$) is the number of vertices with label (resp.~level) equal to $i$.
\end{definition}

If $T$ is a tiered rooted tree, we define $\alpha(T) \coloneqq \alpha(F)$ (resp.~$\beta(T) \coloneqq \beta(F)$), where $F$ is the corresponding forest. Let $\RTT_0(\alpha, \beta) \coloneqq \{ T \in \RTT_0(n) \mid \alpha(T) = \alpha \text{ and } \beta(T) = \beta \}$.

\begin{figure}[!t]
	\centering
	\begin{tikzpicture}[scale=.8]
	\draw[gray!60]
	(-4,0) -- (4,0)
	(-4,2) -- (4,2)
	(-4,4) -- (4,4)
	(-4,6) -- (4,6);
	
	\node at (-5, 0) {$0$};
	\node at (-5, 2) {$1$};
	\node at (-5, 4) {$2$};
	\node at (-5, 6) {$3$};
	
	\node[circle, fill=gray!40, thick, draw=black] (0) at (0,0) {$0$};
	\node[circle, fill=white, thick, draw=black] (1) at (-3,2) {$3$};
	\node[circle, fill=white, thick, draw=black] (2) at (-1,2) {$2$};
	\node[circle, fill=white, thick, draw=black] (3) at (1,2) {$2$};
	\node[circle, fill=white, thick, draw=black] (4) at (3,2) {$1$};
	\node[circle, fill=white, thick, draw=black] (5) at (-2,4) {$4$};
	\node[circle, fill=white, thick, draw=black] (6) at (2,4) {$2$};
	\node[circle, fill=white, thick, draw=black] (7) at (0,6) {$4$};
	\node[circle, fill=white, thick, draw=black] (8) at (-3,6) {$4$};
	
	\draw[red, thick]
	(0) -- (2)
	(0) -- (7)
	(1) -- (5)
	(2) -- (5)
	(3) -- (7)
	(4) -- (6)
	(6) -- (7)
	(1) -- (8);
  \draw[black, very thick, dashed,->] (8)--(5);
	\end{tikzpicture}\quad 
	\begin{tikzpicture}[yscale=.8]
 \tikzset{root/.style={fill=black}, TTnode/.style={circle, draw=black, thick, inner sep=1pt, fill=gray!20}, edge/.style={thick}}
    \node [TTnode, root] (0) at (0,0) {\tiny\color{white}$ 0\,0$};
    \node [TTnode] (1) at (1,1.5) {$\red 2\,\green 1$};
    \node [TTnode] (2) at (-1.5,1.5) {$\red 4\,\green3$};
    \node [TTnode] (3) at (1,3) {$\red 4\,\green2$};
    \node [TTnode] (4) at (1,4.5) {$\red 3\,\green1$};
    \node [TTnode] (5) at (1,6) {$\red 4\,\green3$};
    \node [TTnode] (6) at (-0.5,3) {$\red 2\,\green2$};
    \node [TTnode] (7) at (-0.5,4.5) {$\red 1\,\green1$};
    \node [TTnode] (8) at (-2.5,3) {$\red 2\,\green1$};
    \draw[edge] (0)--(1)--(3)--(4)--(5) (0)--(2)--(6)--(7) (2)--(8);
      \draw[black, very thick, dashed, bend right=45,->] (5) to (3);
    \end{tikzpicture}
	\caption{A tiered rooted tree $T \in \RTT_0((1,3,1,3), (4,2,2))$, represented in two ways. On the left, the tree is drawn as in \cite{DAdderioIraciLeBorgneRomeroVandenWyngaerd2022TieredTrees}: for each $v \in V$, the label $w(v)$ is noted inside the vertex, and $\lv(v)$ is the height of the horizontal line in which $v$ is placed. We will later find it convenient to represent the trees as done on the right, with the ordered pair $(w(v), \lv(v))$ written inside each vertex (the label is first, in red, the level second, in green); note that the height of each vertex does \emph{not} reflect its level. In both representation, one inversion $(u,v)$ is highlighted by a dashed arrow from $v$ to $u$. Note that $w(v)=w(u)=4$, $\lv(v)>\lv(u)$ and $v$ is compatible with $p(u)$.}
	\label{fig:tree}
\end{figure}
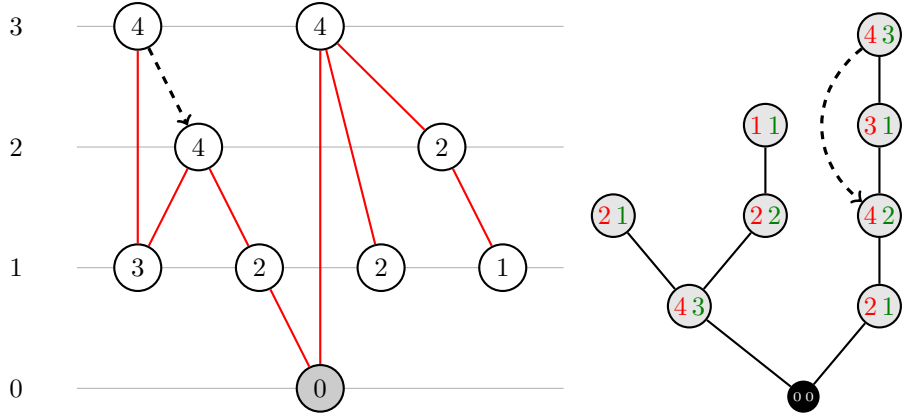

\begin{definition}
	Let $T$ be a rooted tiered tree. Two vertices $u, v$ are said to be \emph{compatible} if either $\lv(u) < \lv(v) \land w(u) < w(v)$, or $\lv(u) > \lv(v) \land w(u) > w(v)$. In this case, we write $u \comp v$.
\end{definition}

Notice that being compatible is not an equivalence relation, as it is symmetric but not transitive.

\begin{definition}
    \label{def:inv}
	Let $T$ be a rooted tiered tree. We say that a pair $(u,v)$ of vertices $u, v \in V \setminus \{r\}$ form an inversion if:
	\begin{enumerate}
		\item $v$ is a descendant of $u$;
		\item $v \comp p(u)$;
		\item either $w(v) < w(u)$, or $w(v) = w(u) \land \lv(v) > \lv(u)$.
	\end{enumerate}
    We define $\inv(T)$ to be the number of inversions of $T$.
\end{definition}

For example, in the tiered rooted tree $T$ in \Cref{fig:tree}, there are $5$ inversions $(u,v)$, whose pairs $((w(u), \lv(u)), (w(v), \lv(v)))$ are $((4,2), (4,3)), ((4,3), (2,2))$, $((4,3), (1,1))$, $((4,3), (2,1))$, and $((2,2), (1,1))$. It follows that $\inv(T) = 5$.

For our purposes, it will be convenient to have a name for the set of trees with no inversions; we set $\zRTT_0(n) \coloneqq \{ T \in \RTT_0(n) \mid \inv(T) = 0 \}$ and $\zRTT_0(\alpha, \beta) \coloneqq \RTT_0(\alpha, \beta) \cap \zRTT_0(n)$.

\subsection{Dyck paths}

\begin{definition}
    A \emph{Dyck path} of size $n$ is a lattice path $\pi$ from $(0,0)$ to $(n,n)$, composed of north and east steps only, that lies entirely weakly above the diagonal $x=y$. A \emph{$2$-labelled Dyck path} is Dyck path equipped with two functions $\pi_x, \pi_y \colon [n] \rightarrow \mathbb{N}_+$ such that the number of east steps of $\pi$ on the line $y = i$ is at least $\chi(\pi_x(i+1) \leq \pi_x(i)) + \chi(\pi_y(i+1) \leq \pi_y(i))$.
\end{definition}

In other words, if we assign labels $\pi_x(i)$ and $\pi_y(i)$ to the north steps of $\pi$, then $\pi_x$ and $\pi_y$ are both strictly increasing along the columns, and if there is exactly one east step between the $i\th$ and $(i+1)\th$ north steps, then at least one between $\pi_x$ and $\pi_y$ is strictly increasing at that point.

Let $\LD^2(n)$ be the set of $2$-labelled Dyck paths of size $n$.

\begin{definition}
    We define the \emph{$x$-composition} (resp.~\emph{$y$-composition}) of a $2$-labelled Dyck path $\pi$ of size $n$ to be the weak composition $\alpha(\pi) \vDash_0 n$ (resp.~$\beta(\pi) \vDash_0 n$) given by $\alpha_i \coloneqq \pi_x^{-1}(i)$ (resp.~$\beta_i \coloneqq \pi_y^{-1}(i)$); that is, $\alpha_i$ (resp.~$\beta_i$) is the number of labels equal to $i$ appearing among the left (resp.~right) labels.
\end{definition}

Let $\LD^2(\alpha, \beta) \coloneqq \{ \pi \in \LD^2(n) \mid \alpha(\pi) = \alpha \text{ and } \beta(\pi) = \beta \}$.

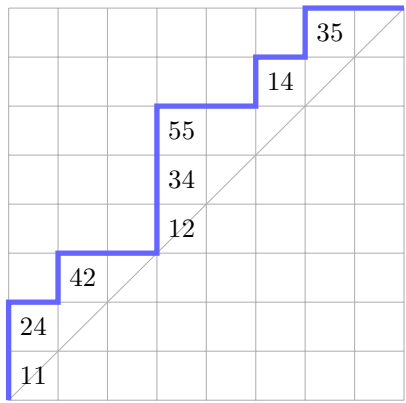
\begin{figure}[!ht]
	\centering
	\begin{tikzpicture}[scale=.65]
		\draw[step=1.0, gray!60, thin] (0,0) grid (8,8);

		\begin{scope}
			\clip (0,0) rectangle (8,8);
			\draw[gray!60, thin] (0,0) -- (8,8);
		\end{scope}

		\draw[blue!60, line width=2pt] (0,0) -- (0,1) -- (0,2) -- (1,2) -- (1,3) -- (2,3) -- (3,3) -- (3,4) -- (3,5) -- (3,6) -- (4,6) -- (5,6) -- (5,7) -- (6,7) -- (6,8) -- (7,8) -- (8,8);

		\node at (0.5,0.5) {$11$};
		\node at (0.5,1.5) {$24$};
		\node at (1.5,2.5) {$42$};
		\node at (3.5,3.5) {$12$};
		\node at (3.5,4.5) {$34$};
		\node at (3.5,5.5) {$55$};
		\node at (5.5,6.5) {$14$};
		\node at (6.5,7.5) {$35$};
	\end{tikzpicture}
	\caption{A 2-labelled Dyck path in $\LD^2((3,1,2,1,1),(1,2,0,3,2))$.}
	\label{fig:labelled-dyck-path}
\end{figure}

\begin{definition}
    For a Dyck path $\pi$, we define $\area(\pi)$ to be the number of whole squares between the path and the main diagonal.    
\end{definition}

For example, the path in \Cref{fig:labelled-dyck-path} has area equal to $7$.

For our puroposes, it will be convenient to have a name for the set of $2$-labelled Dyck path with area equal to $0$, so we let $\zLD^2(n) \coloneqq \{ \pi \in \zLD^2(n) \mid \area(\pi) = 0 \}$, and $\zLD^2(\alpha, \beta) \coloneqq \LD^2(\alpha, \beta) \cap \zLD^2(n)$.

\begin{remark}
	An element of $\zeropaths$ can be identified with a pair of sequences $(a,b) \in (\mathbb{Z}_+^n)^2$ such that, for $i \leq n-1$, we have either $a_i < a_{i+1}$ or $b_i < b_{i+1}$. With a slight abuse of notation, we will sometimes identify the pair of sequences $(a,b)$ with the sequence of pairs $((a_i, b_i))_{i \in [n]}$.
\end{remark}

Finally, we need to define a bijection on weak compositions. For $\beta \vDash_0 n$, let $s(\beta) = \min \{ i \in \N \mid \beta_i \neq 0 \}$ (the first non-zero entry). Define $\rev(\beta) \vDash_0 n$ as $\rev(\beta)_i = \beta_{\ell(\beta) - s(\beta) - i}$, which is essentially the reverse of $\beta$, supported in the same interval. It is clear that $\rev(\rev(\beta)) = \beta$, so $\rev$ is bijective.

\section{Symmetric functions}
\label{sec:symmetricfunctions}

The standard reference for Macdonald polynomials is Macdonald's book \cite{Macdonald1995Book}. For some reference on modified Macdonald polynomials, plethystic substitution, and Delta operators, we have \cite{Haglund2008Book} and \cite{BergeronGarsiaHaimanTesler1999IdentitiesPositivityConjectures}. As a reference for Theta and Xi operators, we have \cite{DAdderioIraciVandenWyngaerd2021ThetaOperators} and \cite{IraciRomero2022DeltaTheta}. Finally, for the super nabla operator, see \cite{BergeronHaglundIraciRomero2023SuperNablaOperator}.

We briefly recall the few definitions we need. Let $\Lambda$ be the algebra of symmetric functions over $\Q(q,t)$; the set $\{ \Ht_\mu[X; q,t] \mu \vdash n, n \in \N \}$ of (modified) Macdonald polynomials is a basis of $\Lambda$ as a vector space.

Set $M=(1-q)(1-t)$, and for any $\mu$ define \[ \Pi_\mu = \prod_{c \in \mu / (1)} \left( 1-q^{a'(c)} t^{l'(c)} \right) \qquad \text{ and } \qquad B_\mu = \sum_{c \in \mu} q^{a'(\mu)} t^{l'(\mu)}, \] where $a'(c)$ and $l'(c)$ denote the \emph{co-arm} and \emph{co-leg} of the cells of $\mu$ (see \Cref{fig:limbs}).

\begin{figure}[!ht]
	\centering
	\begin{tikzpicture}[scale=.52]
		\draw[gray] (0,0) grid (9,1);
		\draw[gray] (0,1) grid (8,2);
		\draw[gray] (0,2) grid (7,3);
		\draw[gray] (0,3) grid (7,4);
		\draw[gray] (0,4) grid (4,5);
		\draw[gray] (0,5) grid (4,6);
		\draw[gray] (0,6) grid (3,7);
		\draw[gray] (0,7) grid (2,8);
		\node at (2.5,3.5) {$c$};
		\fill[blue, opacity=.4] (0,3) rectangle (2,4) node[midway, opacity=1, black]{co-arm};
		\fill[blue, opacity=.4] (3,3) rectangle (7,4) node[midway, opacity=1, black]{arm};
		\fill[blue, opacity=.4] (2,4) rectangle (3,7) node[midway, opacity=1, black, rotate=90]{leg};
		\fill[blue, opacity=.4] (2,3) rectangle (3,0) node[midway, opacity=1, black, rotate=90]{co-leg};
	\end{tikzpicture}
	\caption{Limbs and co-limbs of a cell in a partition.}
	\label{fig:limbs}
\end{figure}
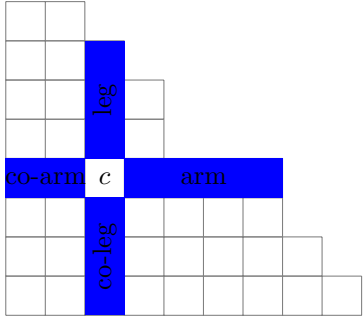

We define the linear operators $\PPi, \Delta_{e_1} \colon \Lambda \rightarrow \Lambda$ on the basis on Macdonald polynomials by

\[ \PPi \Ht_\mu = \Pi_\mu \Ht_\mu \qquad \text{ and } \qquad \Delta_{e_1} \Ht_\mu \coloneqq B_\mu \Ht_\mu; \]

then we can define $\Xi \colon \Lambda \rightarrow \Lambda$ as $\Xi F = M \Delta_{e_1} \PPi F[X/M]$, where the square brackets denote the \emph{plethystic substitution}. Finally, we define $\nabla_\ast \colon \Lambda \rightarrow \Lambda \otimes \Lambda$ by \[ \nabla_\ast \Ht_\mu \coloneqq \Ht_\mu \otimes \Ht_\mu. \]

It is convenient to think of $\Lambda \otimes \Lambda$ as the algebra of symmetric functions on two sets of variables $X$ and $Y$, meaning that they are independently symmetric in each set. With this notation, the specialisation of {\cite[Proposition~5.2]{BergeronHaglundIraciRomero2023SuperNablaOperator}} to $k=1$ states as follows.

\begin{theorem}[{\cite[Proposition~5.2]{BergeronHaglundIraciRomero2023SuperNablaOperator}}]
	\[ \left. \nabla_\ast e_n \right\rvert_{t=1} = \sum_{\pi \in \LD^2(n)} q^{\area(\pi)} x^{\alpha(\pi)} y^{\beta(\pi)}. \]
\end{theorem}

Now, recall the identity \cite[Proposition~5.9]{BergeronHaglundIraciRomero2023SuperNablaOperator}, stating \[ \nabla_\ast e_n = \sum_{\lambda \vdash n} m_\lambda \otimes (\Xi e_\lambda). \] Let us also recall \cite[Conjecture~3]{DAdderioIraciLeBorgneRomeroVandenWyngaerd2022TieredTrees}.

\begin{conjecture}
	\[ \left. \Xi e_\beta \right\rvert_{t=1} = \sum_{T \in \RTT_0(\beta)} q^{\inv(T)} x^{\alpha(T)}, \] where, in our notation, \[\RTT_0(\beta) = \bigcup_{\alpha \vDash_0 \lvert \beta \rvert} \RTT_0(\alpha, \beta). \]
\end{conjecture}

By multiplying each term by the monomial quasisymmetric function $M_\beta$ and taking the sum over $\beta \vDash_0 n$, we can restate the conjecture as follows.

\begin{conjecture}
	\label{conj:tree-conjecture}
	\[ \left. \nabla_\ast e_n \right\rvert_{t=1} = \sum_{T \in \RTT_0(n)} q^{\inv(T)} x^{\alpha(T)} y^{\beta(T)}. \]
\end{conjecture}

In the remainder of the paper, we prove \Cref{conj:tree-conjecture} when $q=0$.

\section{The bijection}
\label{sec:bijection}

This section is dedicated to the proof of our main result, which is the following statement.

\begin{theorem}
	\label{thm:main}
	There exists a bijection $\phi \colon \zerotrees \rightarrow \zeropaths$ such that $\alpha(\phi(T)) = \alpha(T)$ and $\beta(\phi(T)) = \rev(\beta(T))$.
\end{theorem}

\begin{figure}[!ht]
    \centering
    \begin{tikzpicture}
		\tikzset{root/.style={fill=black}, TTnode/.style={circle, draw=black, thick, inner sep=1pt, fill=gray!20}, edge/.style={thick}}
		\node [TTnode, root] (0) at (0,0) {\tiny\color{white}$ 0\,0$};
		\node [TTnode, label=$v_1$] (9) at (-1.5,1) {$\red 3\,\green 3$};
		\node [TTnode, label=$v_2$] (8) at (-0.5,1) {$\red 2\,\green2$};
		\node [TTnode, label=$v_3$] (7) at (0.5,1) {$\red 1\,\green2$};
		\node [TTnode, label=right:$v_4$] (1) at (1.5,1) {$\red 1\,\green3$};
		\node [TTnode, label=left:$v_5$] (2) at (1.5,2) {$\red 4\,\green4$};
		\node [TTnode, label=right:$v_8$] (3) at (2.5,3) {$\red 2\,\green3$};
		\node [TTnode, label=$v_7$] (5) at (1.5,3) {$\red 2\,\green2$};
		\node [TTnode, label=$v_6$] (6) at (0.5,3) {$\red 3\,\green3$};
		\node [TTnode, label=$v_9$] (4) at (2.5,4) {$\red 4\,\green4$};
		\draw[edge] (0)--(1)--(2)--(3)--(4) (5)--(2)--(6) (7)--(0)--(8) (9)--(0);
	\end{tikzpicture}
	\begin{tikzpicture}
		\tikzset{root/.style={fill=black}, TTnode/.style={circle, draw=black, thick, inner sep=1pt, fill=gray!20}, edge/.style={thick}}
		\node [TTnode, root] (0) at (0,0) {\tiny\color{white}$0\,10$};
		\node [TTnode, label=$v_1$] (9) at (-1.5,1) {$\red 3\,\blue 3$};
		\node [TTnode, label=$v_2$] (8) at (-0.5,1) {$\red 2\,\blue4$};
		\node [TTnode, label=$v_3$] (7) at (0.5,1) {$\red 1\,\blue4$};
		\node [TTnode, label=right:$v_4$] (1) at (1.5,1) {$\red 1\,\blue3$};
		\node [TTnode, label=left:$v_5$] (2) at (1.5,2) {$\red 4\,\blue2$};
		\node [TTnode, label=right:$v_8$] (3) at (2.5,3) {$\red 2\,\blue3$};
		\node [TTnode, label=$v_7$] (5) at (1.5,3) {$\red 2\,\blue4$};
		\node [TTnode, label=$v_6$] (6) at (0.5,3) {$\red 3\,\blue3$};
		\node [TTnode, label=$v_9$] (4) at (2.5,4) {$\red 4\,\blue2$};
		\draw[edge] (0)--(1)--(2)--(3)--(4) (5)--(2)--(6) (7)--(0)--(8) (9)--(0);
	\end{tikzpicture}\\$$v_4\prec v_3\prec v_8\prec v_2\prec v_7\prec v_1\prec v_6\prec v_5\prec v_9$$
	\begin{tikzpicture}
		\tikzset{root/.style={fill=black}, TTnode/.style={circle, draw=black, thick, inner sep=1pt, fill=gray!20}, edge/.style={thick}}
		\node [TTnode, root] (0) at (10,0) {\tiny $\color{white} 0\,10$};
		\node [TTnode] (9) at (9,0) {$\red 3\,\blue 3$};
		\node [TTnode] (8) at (8,0) {$\red 2\,\blue4$};
		\node [TTnode] (7) at (7,0) {$\red 1\,\blue4$};
		\node [TTnode] (1) at (6,0) {$\red 1\,\blue3$};
		\node [TTnode] (2) at (5,0) {$\red 4\,\blue2$};
		\node [TTnode] (3) at (2,0) {$\red 2\,\blue3$};
		\node [TTnode] (5) at (3,0) {$\red 2\,\blue4$};
		\node [TTnode] (6) at (4,0) {$\red 3\,\blue3$};
		\node [TTnode] (4) at (1,0) {$\red 4\,\blue2$};
		\draw[thick, bend right=60, dashed, looseness=1] (0) to (1) (0) to (7) (0) to (8) (0) to (9) (1) to (2) (2) to (3) (3) to (4) (2) to (5) (2) to (6);
	\end{tikzpicture}
    \caption{Above, an example of a tree $T\in \zerotrees$ for $n=9$, represented on the left with each vertex $v$ containing the pair $(w(v),\lv(v))$, and on the right the pair $(w(v),\lv'(v))$. The vertices are enumerated in the order $v_1,v_2,\ldots,v_9$ by the exploration from \Cref{def:map}. Below, the total order of \Cref{def:total_order} as it applies to the vertices $v_1,\ldots,v_n$. On the bottom, the pair of sequences $\phi(T)=(a,b)$, in this case given by $((4,2,2,3,4,1,1,2,3),(2,3,4,3,2,3,4,4,3))$; the $i\th$ circle represented contains the pair $(a_i,b_i)=(w(v_{n+1-i}),\lv'(v_{n+1-i}))$.	\label{bij1}}
\end{figure}

\subsection{From trees to sequences}

First, we construct a map $\phi \colon \zerotrees \rightarrow (\mathbb{N}_+^n)^2$.

\begin{definition}
	\label{def:map}
	Given $T \in \zerotrees$, build $\phi(T)$ by visiting its vertices according to a depth-first exploration, where children of a vertex are visited from the largest to the smallest label, and in case of a tie, from the lowest to the highest tier (remember that two sibling vertices cannot have both the same label and the same tier). See \Cref{bij1} for an example. According to this exploration started at the root $r$ of $T$, enumerate the vertices of the tree as $v_0 = r, v_1, \ldots, v_n$. Finally, let $\phi(T) = ((w(v_{n+1-i}))_{i\in [n]}, (\lv'(v_{n+1-i}))_{i\in [n]})$, where $\lv'$ is as defined in \Cref{sec:trees}.
\end{definition}

The following order will be useful.

\begin{definition}
	\label{def:total_order}
	Given $T \in \RTT(n)$ and two vertices $u, v$ of $T$, we say that $u \prec v$ if:
	\begin{enumerate}
		\item $w(u) < w(v)$; or
		\item $w(u) = w(v)$ and $\lv(u) > \lv(v)$; or
		\item $w(u) = w(v)$, $\lv(u) = \lv(v)$, and $u\neq v$ is an ancestor of $v$; or
		\item $w(u) = w(v)$, $\lv(u) = \lv(v)$, $u$ and $v$ belong to different branches, and if $p^a(u) = p^b(v)$ is their least common ancestor, then $p^{b-1}(v) \prec p^{a-1}(u)$.
	\end{enumerate}
   We write $u\preceq v$ to mean that $u=v$ or $u\prec v$.
   
	This defines a total order on the vertices of $T$, such that the children of a vertex $v$ are explored in reverse order with respect to $\prec$. 
\end{definition}

Note that, if we land in condition $4$, then since $p^{a-1}(u)$ and $p^{b-1}(v)$ are siblings, they must have different labels or be in different tiers: in other words, they cannot satisfy condition 3 or condition 4, and can be compared according to $\prec$ by simply checking conditions 1 and 2. Moreover, the last two conditions are equivalent to saying that $u$ is explored before $v$ in $T$ in \Cref{def:map}.

\begin{remark}
	Compare this order with the one from \cite[Definition~3.16]{DAdderioIraciLeBorgneRomeroVandenWyngaerd2022TieredTrees}. They are not the same, but the two orders agree on each branch of $T$, and if two vertices belong to different branches, they cannot possibly form an inversion. It follows that standardising with respect to this order also preserves inversions, so this is also a reading order on the vertices of $T$ and it gives the same fundamental quasisymmetric expansion.
\end{remark}

Let us first state and quickly prove a couple of preliminary results.

\begin{lemma}
	\label{prop:inv_iff_prec}
	If $v$ is a descendant of $u$ such that $v \comp p(u)$, then $(u,v)$ is an inversion if and only if $v \prec u$. 
\end{lemma}

\begin{proof}
	This is clear if $(w(u), \lv(u)) \neq (w(v), \lv(v))$; if instead label and level are both equal, then $(u,v)$ is not an inversion, and indeed $v$ is a descendant of $u$, so $u\prec v$ and thus $v \not \prec u$.
\end{proof}

\begin{lemma}
	\label{lem:prec_not_compatible}
	If $u \prec v$ and $u \not \comp v$, then $\lv(v) \leq \lv(u)$.
\end{lemma}

\begin{proof}
	Since $u \prec v$, either we have $w(u) < w(v)$, or $w(u) = w(v)$ and $\lv(v) \leq \lv(u)$. In the latter case, we are done. If $w(u) < w(v)$, since $u \not \comp v$, we cannot have $\lv(u)<\lv(v)$; thus $\lv(v) \leq \lv(u)$, as wanted.
\end{proof}

We can now show that the map from \Cref{def:map} is indeed a well-defined map from $\zerotrees$ to $\zeropaths$.

\begin{proposition}
	Let $\phi$ be as in \Cref{def:map}. For $T \in \zerotrees$, if $\phi(T) = (a,b)$, then for $i \leq n-1$ we must have $a_i < a_{i+1}$ or $b_i < b_{i+1}$. In particular, we have \[ \phi \colon \zerotrees \rightarrow \zeropaths. \]
	
	Moreover, $\alpha(\phi(T))=\alpha(T)$ and $\beta(\phi(T))=\rev(\beta(T))$.
\end{proposition}

\begin{proof}
	Suppose $\phi(T) = (a, b)$. Given $i \in [n-1]$, the pairs $(a_{n+1-i}, b_{n+1-i})$ and $(a_{n-i},b_{n-i})$ represent the respective label and reverse tier of two vertices $v_i$ and $v_{i+1}$ of $T$ that are visited one immediately after the other ($v_{i+1}$ right after $v_i$) by the depth-first exploration of $T$ that defines $\phi$.
	
	We need to either show that $w(v_{i+1}) < w(v_i)$, or that $\lv(v_i) < \lv(v_{i+1})$; if $w(v_{i+1}) < w(v_i)$ then we are done, so from now on we shall assume $w(v_i) \leq w(v_{i+1})$ and aim to show that $\lv(v_i) < \lv(v_{i+1})$.

	\paragraph{Case 1} $v_{i+1}$ is a descendant of $v_i$. By construction, it must be a child of $v_i$, and therefore such that $v_i \comp v_{i+1}$. The labels of two compatible vertices must be different, so in particular $w(v_i) < w(v_{i+1})$, and since $v_i \comp v_{i+1}$ it follows that $\lv(v_i) < \lv(v_{i+1})$, as desired.

	\begin{figure}
		\centering
		\begin{tikzpicture}
			\tikzset{root/.style={fill=black}, TTnode/.style={circle, draw=black, thick, inner sep=1pt, fill=gray!20}, edge/.style={thick}}
			\fill[gray!10, rounded corners] (-1.5,5.5) rectangle (1,3.5);
			\node [TTnode] (a) at (-1,4) {$\red ?\,\green ?$};
			\node [TTnode] (b) at (0.5,4) {$\red ?\,\green ?$};
			\draw[red, dashed] (a) to node [midway, above] {$\not\comp$} (b);
			\node [TTnode] (c) at (-1,5) {$\red ?\,\green ?$};
			\node [TTnode] (d) at (0.5,5) {$\red ?\,\green ?$};
			\draw[blue, ->] (c) to node [midway, above] {$\prec$} (d);
			\node [TTnode, label=left:{$v_{i}=z$}] (1) at (0,1) {$\red a\,\green b$};
			\node [TTnode, label=above:$v_{i+1}$] (2) at (1,1) {$\red c\,\green d$};
			\node [TTnode, label=left:$v$] (0) at (0,0) {$\red ?\,\green ?$};
			\draw[edge] (2)--(0)--(1);
			\draw[blue,->, bend right] (2) to (1);
			\node at (0,-1) {$a=c, b<d$};
			\node at (0,-2) {\textbf{Case 2.1}};
		\end{tikzpicture}
		\begin{tikzpicture}
			\tikzset{root/.style={fill=black}, TTnode/.style={circle, draw=black, thick, inner sep=1pt, fill=gray!20}, edge/.style={thick}}
			\node [TTnode, label=left:{$v_{i}$}] (1) at (0,1) {$\red a\,\green b$};
			\node [TTnode, label=above:$v_{i+1}$] (2) at (1,-1) {$\red c\,\green d$};
			\node [TTnode, label=left:{$x=z$}] (3) at (0,-1) {$\red e\,\green f$};
			\node (4) at (0,0) {$\ldots$};
			\node [TTnode, label=left:$v$] (0) at (0,-2) {$\red ?\,\green ?$};
			\draw[edge] (2)--(0) (0)--(3)--(4)--(1);
			\draw[blue,->, bend right] (2) to (3);
			\draw[blue,->, bend right] (3) to (1);
			\node at (0,-3) {$c= e=a$}; 
			\node at (0,-3.5) {$b\leq f<d$};
			\node at (0,-4) {\textbf{Case 2.2.1}};
		\end{tikzpicture}
		\begin{tikzpicture}
			\tikzset{root/.style={fill=black}, TTnode/.style={circle, draw=black, thick, inner sep=1pt, fill=gray!20}, edge/.style={thick}}
			\node [TTnode, label=left:{$v_{i}$}] (1) at (0,1) {$\red a\,\green b$};
			\node [TTnode, label=above:$v_{i+1}$] (2) at (1,-1) {$\red c\,\green d$};
			\node [TTnode, label=left:{$z$}] (3) at (0,-1) {$\red e\,\green f$};
			\node (4) at (0,0) {$\ldots$};
			\node [TTnode, label=left:{$x=v$}] (0) at (0,-2) {$\red g\,\green h$};
			\draw[edge] (2)--(0) (0)--(3)--(4)--(1);
			\node at (0,-3) {$g\leq a\leq c$};
			\node at (0,-3.5) {$b\leq h<d$};
			\draw[red, dashed, bend right] (1) to (0);
			\draw[blue, bend left,->] (1) to (3);
			\draw[blue, bend left=60,->] (0) to (1);
			\node at (0,-4) {\textbf{Case 2.2.2}};
		\end{tikzpicture}
		\begin{tikzpicture}
			\tikzset{root/.style={fill=black}, TTnode/.style={circle, draw=black, thick, inner sep=1pt, fill=gray!20}, edge/.style={thick}}
			\node [TTnode, label=above:{$v_{i}$}] (1) at (0,0.3) {$\red a\,\green b$};
			\node [TTnode, label=above:$v_{i+1}$] (2) at (1,-1) {$\red ?\,\green ?$};
			\node [TTnode, label=right:{$z$}] (3) at (0,-1) {$\red ?\,\green ?$};
			\node (4) at (0,-.3) {$\ldots$};
			\node [TTnode, label=right:{$v$}] (0) at (0,-2) {$\red ?\,\green ?$};
			\node (8) at (0,-2.8) {$\ldots$};
			\node [TTnode, label=above right:{$x''$}] (7) at (0,-3.4) {$\red c\,\green d$};
			\node [TTnode, label=above right:{$x'$}] (6) at (0,-4.4) {$\red e\,\green f$};
			\node [TTnode, label=right:{$x$}] (5) at (0,-5.4) {$\red g\,\green h$};
			\draw[edge] (2)--(0) (0)--(3)--(4)--(1) (5)--(6)--(7)--(8)--(0);
			\draw[red, dashed, bend right] (1) to (6);
			\draw[red, dashed, bend right] (1) to (5);
			\draw[blue,->, bend left=90] (1) to (7);
			\draw[blue,->, bend left=90] (1) to (6);
			\draw[blue,->, bend left=45] (5) to (1);
			\node at (0,-6.4) {$h<f\leq b\leq h$};
			\node at (0,-6.9) {$\bot$};
			\node at (0,-7.4) {\textbf{Case 2.2.3}};
		\end{tikzpicture}
		\caption{\label{fig:case2} A summary of the proof that, in the case where $v_{i+1}$ is not a child of $v_i$, the inequality $w(v_i) \leq w(v_{i+1})$ implies $\lv(v_i) < \lv(v_{i+1})$. A red dashed line between two vertices represents the fact that they cannot be compatible; a blue arrow from vertex $u$ to vertex $v$ signifies that $u\prec v$.}
	\end{figure}

		\paragraph{Case 2} $v_{i+1}$ is \emph{not} a descendant of $v_i$ (see \Cref{fig:case2}). The fact that it is visited immediately after $v_i$ implies that it must be a child of some ancestor $v$ of $v_i$. Let $z$ be the child of $v$ that is also a weak ancestor of $v_i$; note that, since $z$ is visited before $v_{i+1}$, we must have $v_{i+1} \prec z$.

		\subparagraph{Case 2.1} $z = v_i$. Since $v_{i+1} \prec z = v_i$, we must have $w(v_{i+1}) \leq w(v_i)$, so equality holds and $\lv(v_i) < \lv(v_{i+1})$, as desired.

		\subparagraph{Case 2.2} $z \neq v_i$. Since $v_{i+1} \prec z$, we have $w(v_i) \leq w(v_{i+1}) \leq w(z)$. Let $x$ be the first ancestor of $v_i$ such that $p(x)$ is also an ancestor of $v_{i+1}$ and $x \prec v_i$ (if no labelled ancestor is found, let $x=r$ be the root of $T$).

		\subparagraph{Case 2.2.1} $x = z$. Then $w(v_{i+1}) \leq w(z)=w(x) \leq w(v_i)$, so since $w(v_i) \leq w(v_{i+1})$ they must all be equal. By definition, $\lv(v_i) \leq \lv(z) < \lv(v_{i+1})$, as desired.

		\subparagraph{Case 2.2.2} $x = v$. Then $w(v) \leq w(v_i) \leq w(v_{i+1})$, but $v$ and $v_{i+1}$ are connected so actually $w(v) < w(v_{i+1})$ and also $\lv(v) < \lv(v_{i+1})$. Now recall that $v_i$ is a descendant of $z$, and notice that, by definition of $x$, $v_i \prec z$. By \Cref{prop:inv_iff_prec}, $v_i \not \comp p(z) = v$ (otherwise $(z, v_i)$ would be an inversion). Now \Cref{lem:prec_not_compatible} implies $\lv(v_i) \leq \lv(v) < \lv(v_{i+1})$, as desired.

		\subparagraph{Case 2.2.3} $x$ is an ancestor of $v$. Let $x'$ be the child of $x$ who is an ancestor of $v_i$ (it can be $v$), and $x''$ its child with the same property (it can be $z$). The vertex $x''$ is an ancestor of $v_i$ such that its parent $x'$ is an ancestor of $v_{i+1}$, so by definition of $x$ we must have $v_i \prec x''$. This implies that $v_i \not \comp p(x'') = x'$, otherwise $(x'', v_i)$ would be an inversion. The same argument applies to $x'$, so we also have $v_i \not \comp x$.
		
		Again by definition of $x$, we have $x \prec v_i \prec x'$, so $w(x) \leq w(x')$; but $x$ and $x'$ are connected, so $w(x) < w(x')$ and thus $\lv(x) < \lv(x')$. Now, by \Cref{lem:prec_not_compatible}, we have both $\lv(x') \leq \lv(v_i)$ (by looking at $v_i \prec x'$) and $\lv(v_i) \leq \lv(x)$ (by looking at $x\prec v_i$), a contradiction.
	\bigskip

	We have thus established that $\phi(T)$ is indeed an element of $\zeropaths$; the fact that $\alpha(\phi(T))=\alpha(T)$ and $\beta(\phi(T))=\rev(\beta(T))$ is obvious by construction.
\end{proof}

The time has now come to prove the injectivity of $\phi$.

\begin{proposition}
	\label{prop:injective}
    The map $\phi \colon \zerotrees \to \zeropaths$ is injective.
\end{proposition}

\begin{proof}
	Let $T, T'$ be tiered trees in $\zerotrees$ such that $\phi(T) = \phi(T')$. The depth-first exploration order given by $\phi$ induces a bijection $f$ between the vertices of $T$ and the vertices of $T'$, and the fact that $\phi(T) = \phi(T')$ implies that, for each vertex $v$ of $T$, $w(f(v)) = w(v)$ and $\lv'(f(v)) = \lv'(v)$. What we wish to show is that, if we know the labels and tiers of the vertices of a tree in $\zerotrees$, as well as the order in which they are visited by the exploration, the tree structure is uniquely determined.

	Indeed, suppose we have determined the tree structure of the subtree on vertices visited up to the $k\th$ step; let $v_0, v_1, \ldots, v_s$ be the ancestry line in $T$ of the $k\th$ visited vertex, where $v_0$ is the root, $v_s$ is the vertex in question and $v_{i+1}$ is a child of $v_i$. All of these vertices have been visited by the exploration up to step $k$. Now, the vertex $v$ visited at step $k+1$, which has label $w(v)$ and reverse tier $\lv'(v)$, must be a child of $v_i$ for some $i \in \{0,\ldots,s\}$. We claim we are able to uniquely determine its parent $p(v)$.

	Suppose that $p(v) = v_i$ for some $i<s$. In this case, $v_{i+1}$ is a sibling of $v$ that has been explored before $v$, so $v \prec v_{i+1}$. Moreover, for all $j \leq i$ such that $v \prec v_j$, since $v$ is a descendant of $v_j$, by \Cref{prop:inv_iff_prec} we must have $v \not \comp p(v_j)$ (or $(v_j,v)$ would be an inversion).

	Consider the sets \[ S = \{0 \leq j < s \mid v \prec v_{j+1} \}\cup\{s\}, \qquad S' = \{0 \leq j \leq s \mid v_j \comp v \}; \]
	by the above, assuming $p(v)=v_i$, we must have $i\in S$, and of course we must also have $i\in S'$; we thus have $S\cap S'\neq \emptyset$. Now, letting $m$ be the minimum of $S \cap S'$, we claim that $m$ is the \emph{only} possible index of $p(v)$. Indeed, if $j > m$ is in $S \cap S'$, then having $p(v)=v_j$ would lead to $v_{m+1}$ and $v$ forming an inversion: we would have $m+1\leq j\leq s$, so $v$ would be a descendant of $v_{m+1}$; moreover, we would have $v \prec v_{m+1}$ (because $m\in S\setminus\{s\}$) as well as $v_m\comp v$ (because $m\in S'$).
	
	This proves that there is only (at most) one possibility for $p(v)$, as required.
\end{proof}

Finally, in order to show that the map $\phi$ is a bijection between $\zerotrees$ and $\zeropaths$, we describe its inverse explicitly.

\subsection{The inverse construction: from sequences to trees}

We now define a map $\psi \colon \zeropaths \rightarrow \zerotrees$, which we will show to be the inverse of $\phi$. 

\begin{definition}
	\label{def:inverse}
    Given $D=(a,b)\in\zeropaths$, build a rooted labelled tree $T=\psi(D)$ as follows. The vertices of $T$ are a root $v_0$, and $v_1,\ldots,v_n$, with $w(v_i)=a_{n+1-i}$, $\lv'(v_i)=b_{n+1-i}$, that is, $\lv(v_i)=L+l-b_{n+1-i}$, where $L=\max_{1\leq i\leq n}b_i$ and $l=\min_{1\leq i\leq n}b_i$. Consider the total order on $v_1,\ldots,v_n$ given by $v_i\prec_* v_j$ if
    \begin{enumerate}
    \item $w(v_i) < w(v_j)$; or
    \item $w(v_i) = w(v_j)$ and $\lv(v_i) > \lv(v_j)$; or
    \item $w(v_i) = w(v_j)$, $\lv(v_i) = \lv(v_j)$, and $i<j$.
    \end{enumerate}
    
    The tree structure is built inductively (see \Cref{fig:inverse}): for $k\geq 0$ (such that $k<n$), the vertices $v_0,\ldots,v_k$ induce a subtree $T_k$ of $T$ such that $v_{i_0}=v_0,v_{i_1},\ldots,v_{i_s}=v_k$ is the ancestral line of $v_k$; given $T_k$, the tree $T_{k+1}\supset T_k$ is obtained by setting the parent of $v_{k+1}$ to be $v_{i_m}$, where
    \[ m = \min \left( \{0 \leq j < s \mid v_{k+1} \comp v_{i_j}, v_{k+1} \prec_* v_{i_{j+1}} \} \cup \{s\} \right). \]
\end{definition}

Clearly, \Cref{def:inverse} does build a tree structure on the vertex set $\{v_0,\ldots,v_n\}$, as the minimum is taken in a nonempty set at each step. Note, in fact, how the construction is essentially the same as the one described in the proof of \Cref{prop:injective}. 

Indeed, the total order introduced in \Cref{def:inverse} is the same as the one from \Cref{def:total_order}, which we now prove. Here we are committing a slight abuse of notation: we have not yet shown $\psi(D)$ to be in $\zerotrees$; however, the order from \Cref{def:total_order} can be defined on any rooted tree with labels $w$ and $\lv$ on the vertices.

\begin{figure}
	\centering
	\begin{tikzpicture}
	\tikzset{root/.style={fill=black}, TTnode/.style={circle, draw=black, thick, inner sep=1pt, fill=gray!20}, edge/.style={thick}}
		\node [TTnode, root, label=below:$v_0$] (0) at (10,0) {\tiny $\color{white} 0\,10$};
		\node [TTnode, label=below:$v_1$] (9) at (9,0) {$\red 3\,\blue 3$};
		\node [TTnode, label=below:$v_2$] (8) at (8,0) {$\red 2\,\blue4$};
		\node [TTnode, label=below:$v_3$] (7) at (7,0) {$\red 1\,\blue4$};
		\node [TTnode, label=below:$v_4$] (1) at (6,0) {$\red 1\,\blue3$};
		\node [TTnode, label=below:$v_5$] (2) at (5,0) {$\red 4\,\blue2$};
		\node [TTnode, label=below:$v_8$] (3) at (2,0) {$\red 2\,\blue3$};
		\node [TTnode, label=below:$v_7$, fill=red!10] (5) at (3,0) {$\red 2\,\blue4$};
		\node [TTnode, label=below:$v_6$] (6) at (4,0) {$\red 3\,\blue3$};
		\node [TTnode, label=below:$v_9$] (4) at (1,0) {$\red 4\,\blue2$};
		\draw[thick, bend right=60, looseness=1] (0) to (1) (0) to (7) (0) to (8) (0) to (9) (1) to (2) (2) to (6);
		\draw[bend left=60, looseness=1,red] (5) to (6) (5) to (2) (5) to (1) (5) to (0);
		\draw[bend left=60, looseness=1,red, ultra thick] (5) to (2);
		\draw[thick, bend right=60, looseness=1, dashed] (3) to (4) (2) to (3);
	\end{tikzpicture}
	\caption{\label{fig:inverse} The 7$\th$ step in the construction of $T=\psi(D)$, where vertices contain their values for $w$ and $\lv'$: the vertex $v_7$ is being added to the tree $T_6$ on $v_0,\ldots,v_6$, whose edges are drawn in black. The possible parents of $v_7$ are the vertices in the ancestry line of $v_6$, that is, $v_6, v_5, v_4, v_0$, which we have joined to $v_7$ with thin red edges. However, we have $v_4\prec_* v_7$ (because $w(v_4)<w(v_7)$), so $p(v_7)\neq v_0$, and $v_7\not\comp v_4$ (because $w(v_7)<w(v_4)$ but $\lv(v_7)>\lv(v_4)$). Instead, we have $p(v_7)=v_5$, because $v_7\comp v_5$ and $v_7\prec_* v_6$. Subsequently added edges are in black, dashed. Note that, if the vertices are placed in the order $v_0,\ldots, v_n$, from right to left, the procedure naturally draws $T$ in a planar way: each step consists in joining $v_{k+1}$ to a parent on the rightmost branch of $T_k$.}
\end{figure}
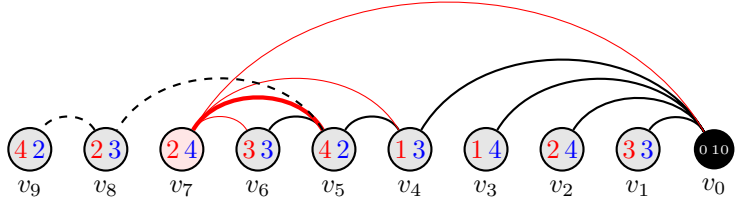

\begin{lemma}
	Given $D=(a,b)\in\zeropaths$, in the labelled tree $T=\psi(D)$ with vertices $v_0,\ldots, v_n$ (numbered as in \Cref{def:inverse}) we have $v_i\prec_* v_j$ if and only if $v_i \prec v_j$, where the latter condition is the one from \Cref{def:total_order}.
\end{lemma}

\begin{proof}
	First of all, if $v_i$ and $v_j$ are siblings in $T$ with parent $v_k$, then $i<j$ if and only if $v_j\prec v_i$ if and only if $v_j\prec_* v_i$; in other words, the children of any vertex $v_k$ are added to $T$ in decreasing order according to both $\prec_*$ and $\prec$ by the procedure of \Cref{def:inverse}.

	Indeed, consider the set $C=\{c<j \mid p(v_c)=v_k\}$ and let $M=\max C$, so that $v_M$ is the latest child of $v_k$ before $v_j$. Since $v_k=p(v_j)$, the vertex $v_k$ belongs to the ancestry line of $v_{j-1}$. But $v_M$ also belongs to the ancestry line of $v_{j-1}$: indeed, it belongs to the ancestry line of $v_l$ for all $l$ between $M$ and $j-1$, as the parent selected for $v_l$ cannot be a strict ancestor of $v_M$ (otherwise we would have an extra child of $v_k$, or $v_k$ itself would not be in the ancestry line of $v_{j-1}$). It follows that, since the parent of $v_j$ is determined to be $v_k$, we have $v_j\prec_*v_M$. Since $j>M$, we must have $w(v_j)<w(v_M)$ or $w(v_j)=w(v_M)$ and $\lv(v_j)>\lv(v_M)$, so we also have $v_j\prec v_M$.

	Now, suppose $i\neq j$ are integers in $[n]$; if $w(v_i)<w(v_j)$ or $w(v_i)=w(v_j)$ and $\lv(v_i)>\lv(v_j)$, then $v_i\prec_* v_j$ and $v_i\prec v_j$.

	Assume from now on that $w(v_i)=w(v_j)$ and $\lv(v_i)=\lv(v_j)$. If $v_i$ is an ancestor of $v_j$, then $i< j$, so again we have both $v_i\prec_* v_j$ and $v_i\prec v_j$. Finally, assume $v_i$ and $v_j$ belong to different branches, and let $p^a(v_i)=p^b(v_j)=v_k$ be their least common ancestor. Set $v_{i'}=p^{a-1}(v_i)$, $v_{j'}=p^{b-1}(v_j)$; the two vertices $v_{i'}$ and $v_{j'}$ are siblings. 

	If $v_i\prec v_j$, which amounts to $v_{j'}\prec v_{i'}$, then we must have $i'<j'$. But this implies that $i<j'<j$ (no descendants $v_l$ of $v_{i'}$ can have index $l> j'$, as $v_{i'}$ would not be in the ancestry line of $v_{l-1}$, and all descendants $v_l$ of $v_{j'}$ have $l>j'$), so we also have $v_i\prec_* v_j$.

	In summary, we have shown that $v_i\prec v_j$ implies $v_i\prec_* v_j$; since the orders are total on $v_1,\ldots, v_n$, this establishes they are the same.
\end{proof}

The fact that the image of $\psi$ is a tiered tree and has zero inversions is now rather easy.

\begin{proposition}
	Given $D=(a,b)\in\zeropaths$, the labelled tree $T=\psi(D)$ is in $\zerotrees$.
\end{proposition}

\begin{proof}
	First of all, we wish to show that for all vertices $v_i$ of $T$ other than the root, $v_i\comp p(v_i)$. This is clear from the construction unless $p(v_i)=v_{i-1}$; however, in the latter case we know that $(w(v_i),\lv'(v_i))=(a_{n+1-i},b_{n+1-i})$ and $(w(p(v_i)),\lv'(p(v_i)))=(a_{n+2-i},b_{n+2-i})$. Consider two cases.

	\paragraph{Case 1} $v_{i-1} \prec v_{i}$. Because $D\in\zeropaths$, if $a_{n+2-i}=w(v_{i-1})<w(v_i)=a_{n+1-i}$ then $b_{n+2-i}>b_{n+1-i}$ and thus $\lv(v_{i-1})<\lv(v_i)$, so indeed $v_{i-1}\comp v_i$. If instead $a_{n+2-i}=w(v_{i-1})=w(v_i)=a_{n+1-i}$, the fact that $v_{i} \prec v_{i-1}$ implies that $\lv(v_{i}) < \lv(v_{i-1})$, so $b_{n+1-i}<b_{n+2-i}$, a contradiction.

	\paragraph{Case 2} $v_i \prec v_{i-1}$. In this case, we know that $v_i \not \comp p(v_{i-1})$. Consider the ancestry line \[ v_0 = v_{j_0}, v_{j_1}, \dots, v_{j_l} = v_{i-1}, \] with $v_{j_{s-1}} = p(v_{j_s})$. Let $x$ be the maximum index such that $j_x=0$ or $v_{j_x} \prec v_i$.

	For $y > x$, by definition we have $v_i \prec v_{j_y}$, so since $p(v_i) \neq p(v_{j_y})$ we must have $v_i \not\comp p(v_{j_y})$. If $x \leq l-2$, then this condition holds for $y = x+1$ and $y = x+2$, so in particular $v_i \not \comp v_{j_x} = p(v_{j_{x+1}})$ and $v_i \not \comp v_{j_{x+1}} = p(v_{j_{x+2}})$.

	We have $v_{j_x} \prec v_i \preceq v_{j_{x+1}}$, but $v_{j_x}$ and $v_{j_{x+1}}$ are connected, so $a_{j_x} < a_{j_{x+1}}$ and $b_{j_x} > b_{j_{x+1}}$. By \Cref{lem:prec_not_compatible}, we must have $b_{j_x} \leq b_i \leq b_{j_{x+1}}$, a contradiction.

	It follows that $x=l-1$, so $p(v_{i-1}) = v_{j_{l-1}} \prec v_i$. Again by \Cref{lem:prec_not_compatible} we get $b_{j_{l-1}} \leq b_i$; but $v_{j_{l-1}} \prec v_i \prec v_{i+1}$ and $p(v_{i-1})$ is connected to $v_{i-1}$, so $a_{j_{l-1}} < a_{i-1}$ and $b_{j_{l-1}} > b_{i-1}$. In particular, $b_{i-1} < b_{j_{l-1}} \leq b_i$, so $v_i \comp v_{i-1}$, as desired.

	Finally, $T$ has no inversions by construction: since it is a tiered tree, \Cref{prop:inv_iff_prec} yields that any pair $(u,v)$, where $v$ is a descendant of $u$, cannot form an inversion because either $v \not\comp p(u)$ or $v \not\prec u$.
\end{proof}

The fact that $\psi \colon \zeropaths \to \zerotrees$ is the inverse of $\phi \colon \zerotrees \to \zeropaths$ is now quite clear.

\begin{lemma}
	\label{lem:inverse}
	For $D \in \zeropaths$, $\phi(\psi(D)) = D$.
\end{lemma}

\begin{proof}
	Let $D = (a,b)$. It is enough to show that, if $T = \psi(D)$ and $v_i$ is the vertex of $T$ corresponding to $(a_{n+1-i}, b_{n+1-i})$, then the vertices of $T$ are explored in the order $v_0 = r, v_1, \dots, v_n$. In other words, it is enough to show that, for any $i < j$, $v_i$ comes before $v_j$ in the depth-first exploration of \Cref{def:map}.

	This, however, is clear from the construction $\psi$. If $i < j$, then either $v_j$ is in the tree of descendants of $v_i$, in which case it does come after it in the exploration, or it is within the tree of descendants of a sibling $v_k$ of $v_i$. Note that it cannot be an ancestor of $v_i$, as those correspond to pairs appearing on the right of $(a_{n+1-j}, b_{n+1-j})$ in $(a,b)$. If $v_k \prec v_i$, then $v_k$ comes after $v_i$ in the exploration, and since $v_j$ is a descendant of $v_k$ so does $v_j$, as desired. If $v_k \succ v_i$ instead, then by construction all descendants of $v_k$ are of the form $v_h$ with $k < h < i$, which contradicts $i < j$. It follows that $v_i$ comes before $v_j$ in the exploration.
 \end{proof}

\subsection{Proof of \texorpdfstring{\Cref{thm:main}}{the main theorem}}

We now have all the ingredients to conclude the proof of \Cref{thm:main}. For all $\alpha, \beta \vDash n$, the maps $\phi$ and $\psi$ restrict to maps \[ \phi \colon \zRTT_0(\alpha, \beta) \rightarrow \zLD^2(\alpha, \rev(\beta)), \quad \psi \colon \zLD^2(\alpha, \beta) \rightarrow \zRTT_0(\alpha, \rev(\beta)), \] which are finite sets. Moreover, $\phi$ is injective, and (up to replacing $\beta$ with $\rev(\beta)$ in one of the two maps, and recalling that $\rev$ is an involution) we have that $\phi \circ \psi$ is the identity, so $\psi$ is also injective. This implies that both maps are bijections, which is what we wanted to show.

\bibliographystyle{plain}
\bibliography{references}

\begin{thebibliography}{10}

\bibitem{BergeronGarsiaHaimanTesler1999IdentitiesPositivityConjectures}
Fran\c{c}ois Bergeron, Adriano~M. Garsia, Mark Haiman, and Glenn Tesler.
\newblock Identities and positivity conjectures for some remarkable operators
  in the theory of symmetric functions.
\newblock {\em Methods Appl. Anal.}, 6(3):363--420, 1999.
\newblock Dedicated to Richard A. Askey on the occasion of his 65th birthday,
  Part III.

\bibitem{BergeronHaglundIraciRomero2023SuperNablaOperator}
Fran{\c{c}}ois {Bergeron}, James {Haglund}, Alessandro {Iraci}, and Marino
  {Romero}.
\newblock {The super nabla operator}.
\newblock {\em arXiv e-prints}, page arXiv:2303.00560, March 2023.

\bibitem{CarlssonMellit2018ShuffleConjecture}
Erik Carlsson and Anton Mellit.
\newblock A proof of the shuffle conjecture.
\newblock {\em J. Amer. Math. Soc.}, 31(3):661--697, 2018.

\bibitem{DAdderioIraciVandenWyngaerd2021ThetaOperators}
Michele D'Adderio, Alessandro Iraci, and Anna Vanden~Wyngaerd.
\newblock Theta operators, refined {Delta} conjectures, and coinvariants.
\newblock {\em Adv Math}, 376:107447, January 2021.

\bibitem{DAdderioMellit2022CompositionalDelta}
Michele D'Adderio and Anton Mellit.
\newblock A proof of the compositional {Delta} conjecture.
\newblock {\em Advances in Mathematics}, 402:108342, June 2022.

\bibitem{DuganGlennonGunnellsSteingrimsson2019TieredTreesqEulerian}
William Dugan, Sam Glennon, Paul~E. Gunnells, and Einar Steingrímsson.
\newblock Tiered trees, weights, and q-{E}ulerian numbers.
\newblock {\em J. Combin. Theory Ser. A}, 164:24--49, 2019.

\bibitem{DukesSeligSmithSteingrimsson2018TieredTrees}
Mark Dukes, Thomas Selig, Jason~P. Smith, and Einar Steingrimsson.
\newblock Permutation graphs and the {Abelian} sandpile model, tiered trees and
  non-ambiguous binary trees.
\newblock {\em ArXiv e-prints}, October 2018.
\newblock arXiv: 1810.02437.

\bibitem{DAdderioIraciLeBorgneRomeroVandenWyngaerd2022TieredTrees}
Michele D’Adderio, Alessandro Iraci, Yvan Le~Borgne, Marino Romero, and Anna
  Vanden~Wyngaerd.
\newblock {Tiered Trees and Theta Operators}.
\newblock {\em International Mathematics Research Notices}, 09 2022.
\newblock rnac258.

\bibitem{GunnellsLetellierRodriguezVillegas2018Quivers}
Paul~E. Gunnells, Emmanuel Letellier, and Fernando Rodriguez~Villegas.
\newblock Torus orbits on homogeneous varieties and {Kac} polynomials of
  quivers.
\newblock {\em Math. Z.}, 290(1):445--467, October 2018.

\bibitem{Haglund2008Book}
James Haglund.
\newblock {\em The {$q$},{$t$}-{C}atalan numbers and the space of diagonal
  harmonics}, volume~41 of {\em University Lecture Series}.
\newblock American Mathematical Society, Providence, RI, 2008.
\newblock With an appendix on the combinatorics of Macdonald polynomials.

\bibitem{HaglundHaimanLoehrRemmelUlyanov2005ShuffleConjecture}
James Haglund, Mark Haiman, Nicholas Loehr, Jeffrey~B. Remmel, and Anatoly
  Ulyanov.
\newblock A combinatorial formula for the character of the diagonal
  coinvariants.
\newblock {\em Duke Math J}, 126(2):195--232, 2005.

\bibitem{HaglundRemmelWilson2018DeltaConjecture}
James Haglund, Jeffrey~B. Remmel, and Andrew~T. Wilson.
\newblock The {D}elta {C}onjecture.
\newblock {\em Trans. Amer. Math. Soc.}, 370(6):4029--4057, 2018.

\bibitem{IraciRomero2022DeltaTheta}
Alessandro Iraci and Marino Romero.
\newblock Delta and theta operator expansions.
\newblock {\em Forum of Mathematics, Sigma}, 12:e30, 2024.

\bibitem{Macdonald1995Book}
Ian~G. Macdonald.
\newblock {\em Symmetric functions and {H}all polynomials}.
\newblock Oxford Mathematical Monographs. The Clarendon Press, Oxford
  University Press, New York, second edition, 1995.
\newblock With contributions by A. Zelevinsky, Oxford Science Publications.

\bibitem{Postnikov1997IntransitiveTrees}
Alexander Postnikov.
\newblock Intransitive {T}rees.
\newblock {\em J. Combin. Theory Ser. A}, 79(2):360--366, 1997.

\end{thebibliography}

\end{document}